\documentclass[11pt,reqno]{amsart}

\usepackage{a4wide}
\usepackage{amsmath}
\usepackage{amsthm}
\usepackage{amsfonts}
\usepackage{amssymb}
\usepackage{xcolor}
\usepackage{mathtools}
\usepackage{eucal}
\usepackage{mathrsfs}

\usepackage{hyperref}

\usepackage{graphicx}

\usepackage{todonotes}
\usepackage{eucal}
\usepackage[normalem]{ulem}
\usepackage{cancel}
\usepackage{bm}

\theoremstyle{plain}
\newtheorem{thm}{Theorem}[section]
\newtheorem{lem}[thm]{Lemma}
\newtheorem{cor}[thm]{Corollary}

\theoremstyle{definition}
\newtheorem{rem}[thm]{Remark}

\DeclareMathOperator{\N}{\mathbb N}

\DeclareMathOperator{\D}{\mathbf D}
\newcommand{\RpN}{\mathbb{R}_+^\mathbb{N}}

\numberwithin{equation}{section}

\allowdisplaybreaks

\begin{document}

\author{Tu\u{g}\c{c}e \"{U}nver, Amiran Gogatishvili, Nurzhan Bokayev \and Nurgul Kuzeubayeva}

\title{Weighted Hardy inequalities involving supremum for decreasing sequences}

\address{T. \"{U}nver,
Department of Mathematics, Kirikkale University, 71450 Yahsihan, Kirikkale, T\" urkiye}
\email{tugceunver@kku.edu.tr}

\address{A. Gogatishvili,
Institute of Mathematical of the  Czech Academy of Sciences, \v Zitn\'a~25, 115~67 Praha~1, Czech Republic}
\email{gogatish@math.cas.cz}

\address{N. Bokayev,
Department of Fundamental Mathematics, L.N. Gumilyov Eurasian National University,2, Satbaev St., 010000, Astana, Kazakhstan}
\email{bokayev2011@yandex.ru}

\address{N. Kuzeubayeva,  
Department of Fundamental Mathematics, L.N. Gumilyov Eurasian National University, 2, Satbaev St., 010000, Astana, Kazakhstan}
\email{nurgul.kuzeubaeva@mail.ru}



\thanks{The research was supported by a grant from the Ministry of Education and Science of the Republic of Kazakhstan (project no. AP26196065). The research of  A.~Gogatishvili was partially supported by the grant project 23-04720S of the Czech Science Foundation (GA\v{C}R), The Institute of Mathematics, CAS is supported by RVO:67985840, by  Shota Rustaveli National Science Foundation (SRNSF), grant no: FR22-17770.}
\begin{abstract}
In this paper, we provide a complete characterization of the weighted Hardy inequalities involving the supremum operator,  restricted to the cone of non-increasing sequences, for all positive parameters. We reduce such inequalities to equivalent ones on the cone of non-negative sequences. The latter setting provides a broader framework for analysis and significantly expands the range of proofs that can be established. 
\end{abstract}

\maketitle


\section{Introduction}

Let $0< p, q < \infty$ be fixed parameters, and let $\bm{u}= \{u(m)\}_{m\in \mathbb{N}}$, $\bm{a}=\{a(m)\}_{m\in \mathbb{N}}$, and $\bm{b}=\{b(m)\}_{m\in \mathbb{N}}$ be given sequences of non-negative real numbers.  Let us denote by $\RpN$ the space of all sequences of non-negative real numbers. In this paper, we study the weighted inequality
\begin{equation}  \label{eq1}
	\bigg(\sum_{m=1}^{\infty}
	\bigg(\sup_{m\le i<\infty} u(i)\sum^i_{k=1}x(k)\bigg)^q a(m)\bigg)^{\frac{1}{q}} \le C
	\bigg(\sum^\infty_{m=1} x(m)^p b(m)\bigg)^{\frac{1}{p}}
\end{equation}
for all non-negative, non-increasing sequences $\bm{x}=\{x(m)\}_{m\in \mathbb{N}}$. Our goal is to identify the necessary and sufficient conditions on the sequences $\bm{u}, \bm{a}, \bm{b}\in \RpN$ under which inequality \eqref{eq1} is valid for some constant $C > 0$, independent of the choice of $\bm{x}$. The characterization of inequality \eqref{eq1} for arbitrary non-negative sequences $\bm{x}$ without any monotonicity assumption was obtained in \cite{GogKrep}. 

Note that, for a non-increasing sequence $\bm{x}$, the sequence $\big\{\frac{1}{i} \sum_{k=1}^i x(k)\big\}_{i\in \mathbb{N}}$ is also non-increasing. Therefore, for any $m\geq 1$, we have 
\[ \sup_{m\le i<\infty} \frac{1}{i}\sum_{k=1}^{i} x(k)=\frac{1}{m} \sum_{k=1}^m x(k).\]
In particular, if we choose $u(i)=\frac{1}{i}$ in \eqref{eq1}, we recover the discrete weighted Hardy inequality. The characterization of weighted Hardy inequalities for decreasing sequences was obtained by Bennett and Grosse-Erdmann in \cite{Bennett Gross 1}. Consequently, our approach yields an alternative proof of their result and, moreover, naturally extends it to the non-linear case. Specifically, our approach allows for the treatment of the supremum operator, which does not satisfy linearity.

For the discrete weighted  Hardy inequality, in the case $1 < p, q < \infty$, an effective approach is based on Sawyer's duality principle \cite{Sawyer}, which reduces inequalities of the form \eqref{eq1} for non-negative, non-increasing sequences to corresponding modified inequalities for non-negative sequences. The discrete analog of Sawyer's duality theorem was obtained by Oinarov and Shalginbaeva \cite{OinarShalg} and Bennett and Grosse-Erdmann \cite{Bennett Gross, Bennett Gross 1}. However, Sawyer's duality principle cannot be applied to the supremum operator appearing in \eqref{eq1}, since this operator is non-linear. We therefore develop a new reduction method that is valid for all parameters  $0<p,q<\infty$ and is applicable to non-linear operators.  The inclusion of this class of operators is essential, as supremal operators naturally emerge when investigating fractional maximal operators on classical Lorentz spaces \cite{CKOP,GOP}. In fact, establishing inequality \eqref{eq1} constitutes a fundamental step in analyzing the boundedness of the discrete fractional maximal operator on Lorentz sequence spaces. We intend to address this in our next study.

Reduction theorems for the weighted Hardy inequality involving the supremum for non-increasing functions in the continuous setting, namely, for inequalities involving functions rather than sequences, were obtained in \cite[Corollary~3.2]{GogPick-nrt} for $0<p\le 1$ and $0<q<\infty$. For $1 \leq p < \infty$ and $0<q<\infty $, a general reduction theorem for monotone operators was obtained in \cite{GogStep1} and \cite{GogStep2}, which contains the Hardy operator involving the supremum as a special case. A complete characterization in the continuous setting is provided in \cite{GogMus-MIA}. For a detailed history of related results, see the papers cited therein.  

In the case $1<p < \infty$, a fundamental difference arises between the discrete and continuous settings. The proof used in the continuous case cannot be applied directly to the discrete one, since the power rule, which holds in the continuous case, fails in the discrete setting (see \cite{Sinnamon}). Moreover, our approach requires a new version of Copson's inequality (Lemma~4.6), which was stated without proof in \cite{BGKKT}; for completeness, we provide the proof in this paper. In the case $0<p<1$, our proofs are essentially based on the results from \cite{GogKrep}.

Throughout the paper, we denote by $\N$ the set of natural numbers. For non-negative quantities $A$ and $B$, we write $A\lesssim B$ if $A \leq C B$ for some positive constant $C$, and $A\approx B$ if both $A\lesssim B$ and $B\lesssim A$. A \textit{weight} will always refer to a sequence of non-negative real numbers. We adopt the conventions: $0 \cdot \infty = 0$, $0/0 = 0$, $\infty/\infty = 0$, and $1/0 = \infty$ with $1/\infty = 0$. Furthermore, for any $a>0$, we set $\infty^a = \infty$, and for any $a<0$, we set $\infty^\alpha = 0$.

The paper is organized as follows. Section \ref{S3} is devoted to the main result. In Section \ref{S2}, we present a reduction theorem for weighted inequalities for supremum operators involving the Hardy operator on the cone of non-increasing sequences, while in Section \ref{S4}, we collect auxiliary results concerning discrete Hardy-type operators. Finally, the proofs of the main theorems from Sections \ref{S3} and \ref{S2} are provided in Section \ref{S5}. 

\section{Main Result} \label{S3} 

In this section, we present a complete characterization of inequality~\eqref{eq1}. To this end, we shall employ the following notation. We assume throughout that $\bm{a}$ and $\bm{b}$ are sequences of non-negative terms. Their partial sums are denoted by
\[
A(m) =\sum^m_{k=1}a(k), \quad B(m) =\sum^m_{k=1}b(k), \quad m\in {\mathbb N}.
\] 
We define
\[
B(\infty) = \lim_{m \to \infty} B(m).
\]

Let $\bm{u}$ be a given non-negative sequence, we define 
\begin{equation}\label{def:z(m)}
	z(m) := \sup_{m\le k < \infty} u(k),
\end{equation}
and the sequence $\bm{z}=\{z(m)\}$ is called the \emph{decreasing upper envelope of} $\bm{u}$.

\begin{thm} \label{Th0}
Let $0<p, q<\infty$. Assume that $\bm{a}$, $\bm{b}$ and $\bm{u}$ are given non-negative sequences with $\,b(1)>0$ and $B(\infty) < \infty$. Then inequality \eqref{eq1} holds  for all non-negative, non-increasing sequences $\bm{x}$ if and only if
\begin{equation} \label{eq0}
C_0 :=\bigg(\sum_{m=1}^\infty\bigg(\sup_{m\le i<\infty} u(i) i \bigg)^qa(m)\bigg)^{\frac{1}{q}}  B(\infty)^{-\frac{1}{p}}<+\infty.
\end{equation}
Moreover, if  $C$ denotes the best constant in \eqref{eq1},  then 
\begin{equation}\label{C0-estimate}
C_0\le C \leq C_0
b(1)^{-\frac{1}{p}} B(\infty)^{\frac{1}{p}}.
\end{equation}
\end{thm}

\begin{rem}
It should be noted that $C_0$ is not equivalent to the best constant in \eqref{eq1}, as the upper estimate in \eqref{C0-estimate} depends on $\bm{b}$. 
\end{rem}
In the case where $B(\infty) = \infty$, the constant $C_0$ is zero by convention, consistent with the definition in \eqref{eq0}.

\begin{thm} \label{thmmain}
Let $0<p, q<\infty$. Assume that $\bm{a}, \bm{b}$, and $\bm{u}$ are given non-negative sequences with $b(1)>0$. Then inequality \eqref{eq1} holds for all non-negative, non-increasing sequences $\bm{x}$ if and only if 
	
{\textup{ (i)}} $1 < p \le q$ and $\max\{C_0, C_1, C_2, C_3\} < \infty$, where
\begin{align*}
C_1 & := \sup_{1\le n<\infty}  z(n)A(n)^{\frac{1}{q}}\Bigg (\sum_{m=1}^{n} \frac{ B(m)^{\frac{1}{1-p}} b(m+1)}{m^{\frac{p}{1-p}} B(m+1)}\Bigg)\sp{\frac{p-1}{p}},\\
C_2 & := \sup_{1\le n<\infty} \Bigg( \sum_{i=n}^{\infty} z(i)^q a(i) \Bigg)^{\frac{1}{q}} \Bigg (\sum_{m=1}^{n} \frac{ B(m)^{\frac{1}{1-p}} b(m+1)}{m^{\frac{p}{1-p}} B(m+1)}\Bigg)\sp{\frac{p-1}{p}},\\
C_3 & := \sup_{1\le n<\infty} \Bigg( \sum_{i=1}^{n} a(i) \sup_{i\le j\le n} u(j)^q j^q \Bigg)^{\frac{1}{q}} B(n)^{-\frac{1}{p}}.
\end{align*}
	
{\textup{ (ii)}} $1 < p$, $q<p$ and $\max\{C_0, C_4, C_5, C_6, C_7\} < \infty$, where
\begin{align*}
C_4 &  := \Bigg( \sum_{n=1}^{\infty}  \Bigg(\sum_{i= n}^{\infty} z(i)^q a(i)\Bigg)^{\frac{q}{p-q}} z(n)^{q} a(n) 
\Bigg(\sum_{m=1}^{n}\frac{ B(m)^{\frac{1}{1-p}} b(m+1)}{m^{\frac{p}{1-p}} B(m+1)}\Bigg) ^{\frac{(p-1)q}{p-q}} \Bigg)^\frac{p-q}{pq},\\
C_5 & := \Bigg( \sum_{n=1}^{\infty} A(n)^{\frac{q}{p-q}} a(n) 
\sup_{n\le k < \infty} z(k)^{\frac{pq}{p-q}} \Bigg(\sum_{m=1}^{k} \frac{ B(m)^{\frac{1}{1-p}} b(m+1)}{m^{\frac{p}{1-p}} B(m+1)}\Bigg)^{\frac{(p-1)q}{p-q}} \Bigg)^\frac{p-q}{pq},\\
C_6 & := \Bigg( \sum_{n=1}^{\infty} A(n)^{\frac{q}{p-q}} a(n) \sup_{n\le k < \infty} u(k)^{\frac{pq}{p-q}} k^{\frac{pq}{p-q}} B(k)^{-\frac{q}{p-q}}\Bigg)^\frac{p-q}{pq} ,\\
C_7	& := \Bigg( \sum_{n=1}^{\infty} \Bigg(\sum_{i=1}^n a(i) \sup_{i\le j\le n} u(j)^q j^q \Bigg)^{\frac{q}{p-q}} a(n) \sup_{n \le k <\infty} u(k)^q k^q B(k)^{-\frac{q}{p-q}} \Bigg)^\frac{p-q}{pq}.
\end{align*}
	
{\textup{(iii)}} $0<p\le 1$, $p\le q$ and $\max\{C_8, C_9, C_{10}\} < \infty$, where
\begin{align*}
C_8 & := \sup_{1\le n<\infty} z(n) A(n)^{\frac 1q} \sup_{1\le m\le n} m B(m)^{-\frac{1}{p}}, \\
C_9 & := \sup_{1\le n<\infty}\Bigg( \sum_{k=n}^{\infty} z(k)^q a(k) \Bigg)^{\frac 1q} \sup_{1\le m\le n} m B(m)^{-\frac{1}{p}},\\
C_{10} & := \sup_{1\le n<\infty}\Bigg( \sum_{i=1}^n a(i)\sup_{i\le j\le n}u(j)^q j^q  \Bigg)^{\frac 1q} B(n) ^{- \frac{1}{p}}.
\end{align*}
	
{\textup{ (iv)}} $0<q<p\le 1$ and $\max\{C_{11}, C_{12}, C_{13}\} < \infty$, where
\begin{align*}	
C_{11} & :=  \Bigg( \sum_{n=1}^{\infty} \Bigg(\sum_{i=n}^{\infty} z(i)^q a(i) \Bigg)^{\frac{q}{p-q}}	z(n)^{q} a(n) \sup_{1\le k\le n} k^{\frac{pq}{p-q}}B(k)^{-\frac{q}{p-q}}\Bigg)^{\frac{p-q}{pq}},\\
C_{12}& := \Bigg(\sum_{n=1}^{\infty} A(n)^{\frac{q}{p-q}} a(n) \sup_{n\le k <\infty} z(k)^{\frac{pq}{p-q}} k^{\frac{pq}{p-q}}B(k)^{-\frac q{p-q}}\Bigg)^{\frac{p-q}{pq}},\\
C_{13} & :=\Bigg( \sum_{n=1}^{\infty} \Bigg(\sum_{i=1}^n a(i) \sup_{i\le j\le n} u(j)^q j^q \Bigg)^{\frac{q}{p-q}} a(n) \sup_{n \le k < \infty} u(k)^q k^q   B(k)^{-\frac{q}{p-q}} \Bigg)^{\frac{p-q}{pq}}.
\end{align*}
Moreover, if $C$ denotes the best constant in~\eqref{eq1}, then
\begin{equation*}
C \approx
\begin{cases}
C_0 + C_1+C_2+C_3 &\text{in the case \textup{(i)},} \\
C_0 + C_4 + C_5 + C_6 + C_7  &\text{in the case \textup{(ii)},} \\
C_8 + C_9 + C_{10}  &\text{in the case \textup{(iii)},} \\
C_{11}+ C_{12}+ C_{13}  &\text{in the case \textup{(iv)},}
\end{cases}
\end{equation*}
where the multiplicative constants in all the equivalences above depend only on $p$ and $q$.
\end{thm}

\section{Reduction theorems for weighted inequalities on the cone of non-increasing sequences} \label{S2}

This section is devoted to the presentation of the reduction theorems. 

\begin{thm} \label{Th1}
Let $0<q<\infty$ and $1 < p<\infty$. Assume that $\bm{a}, \bm{b}$ and $\bm{u}$ are given non-negative sequences with $b(1)>0$. Then inequality \eqref{eq1} holds for all non-negative, non-increasing sequences $\bm{x}$ if and only if $C_0<+\infty$ and the inequality 
\begin{align}\label{eq2}
\bigg(\sum_{m=1}^\infty\bigg(\sup_{m\le i<\infty} u(i) \sum_{k=1}^i \sum_{j=k}^\infty y(j) \bigg)^q a(m)\bigg)^{\frac{1}{q}}\notag\\
&\hskip-4cm\leq \mathcal{D} \bigg(\sum_{m=1}^\infty y(m)^pB({m+1})^{p-1}B(m) b({m+1})^{1-p}\bigg)^{\frac{1}{p}}  
\end{align}
holds for all non-negative sequences $\bm{y}=\{y(m)\}_{m \in \mathbb{N}}$.  Moreover, if $C$ and $\mathcal{D}$ denote the best constants in \eqref{eq1} and \eqref{eq2}, respectively, then $C\approx \mathcal{D}+ C_0$, where $C_0$ is defined in \eqref{eq0} and the implicit multiplicative constants in the equivalence depend only on $p$ and $q$.
\end{thm}

\begin{thm} \label{Th2}
Let $0<q<\infty$ and $0<p \le 1$. Assume that $\bm{a}, \bm{b}$ and $\bm{u}$ are given non-negative sequences with $b(1)>0$. Then the following six statements are equivalent:
	
{\textup{(i)}}  Inequality \eqref{eq1} holds  for all non-negative, non-increasing sequences $\bm{x}$.

{\textup{(ii)}}  Inequality
\begin{equation}\label{eq3}
\bigg(\sum_{m=1}^\infty\bigg( \sup_{m\le i<\infty}  u(i) \sum_{k=1}^i \bigg(\sum_{j=k}^\infty y(j)\bigg)^{\frac{1}{p}} \bigg)^q a(m)\bigg)^{\frac{1}{q}}\leq \D_1\bigg(\sum_{m=1}^\infty y(m) B(m) \bigg)^{\frac{1}{p}} 
\end{equation}
holds for all non-negative sequences $\bm{y}$.
	
{\textup{(iii)}}  Inequality
\begin{equation} \label{eq4}
\bigg(\sum_{m=1}^\infty\bigg(\sup_{m\le i<\infty } u(i)^p \sum_{k=1}^i k^{p-1} \sum_{j=k}^\infty y(j) \bigg)^{\frac{q}{p}} a(m)\bigg)^{\frac{1}{q}}\leq \D_2\bigg(\sum_{m=1}^\infty y(m) B(m) \bigg)^{\frac{1}{p}} 
\end{equation}
holds for all non-negative sequences $\bm{y}$. 
	
{\textup{(iv)}}  Inequality
\begin{equation} \label{eq5}
\bigg(\sum_{m=1}^\infty \bigg(\sup_{m \le i<\infty } u(i)^p \sup_{1 \le k \le i} k^{p}  \sum_{j=k}^\infty y(j) \bigg)^{\frac{q}{p}} a(m)\bigg)^{\frac{1}{q}}\leq \D_3\bigg(\sum_{m=1}^\infty y(m) B(m) \bigg)^{\frac{1}{p}} 
\end{equation}
holds for all non-negative sequences $\bm{y}$.
	
{\textup{(v)}}  Inequality
\begin{equation} \label{eq6}
\bigg(\sum_{m=1}^\infty\bigg(\sup_{m\le i<\infty } 	u(i)\sum_{k=1}^i \bigg(\sup_{k\le j< \infty} y(j)\bigg) ^{\frac{1}{p}}\bigg)^{q} a(m)\bigg)^{\frac{1}{q}}\leq \D_4\bigg(\sum_{m=1}^\infty y(m) B(m) \bigg)^{\frac{1}{p}} 
\end{equation}
holds for all non-negative sequences $\bm{y}$.
	
{\textup{(vi)}}  Inequality
\begin{equation} \label{eq7}
\bigg(\sum_{m=1}^\infty \bigg( \sup_{m \le i<\infty } u(i)^p \sup_{1 \le k \le i} k^{p}  \sup_{k \le j < \infty} y(j) \bigg)^{\frac{q}{p}} a(m)\bigg)^{\frac{1}{q}}\leq \D_5 \bigg(\sum_{m=1}^\infty y(m) B(m) \bigg)^{\frac{1}{p}} 
\end{equation}
holds for all non-negative sequences $\bm{y}$.
	
Moreover, if  $C$ and  $\D_i$,  $i=1,\ldots,5$  denote the best constants in \eqref{eq1} and \eqref{eq3} -- \eqref{eq7}, respectively, then $C\approx \D_i$, $i=1,\ldots,5$, where the implicit multiplicative constants in the equivalence depend only on $p$ and $q$.
\end{thm}

\begin{cor}  \label{cor2.7}
Let $0 < p, q < \infty$. Assume that $\bm{a}, \bm{b}$ and $\bm{u}$ are given non-negative sequences with $b(1)>0$. Then inequality \eqref{eq1} holds  for all non-negative, non-increasing sequences $\bm{x}$ if and only if: 
	
{\textup{(i)}}  $1 < p < \infty$, $C_0<\infty$ and inequalities
\begin{align} 
&\bigg(\sum_{m=1}^{\infty}
\bigg(\sup_{m \le i < \infty} u(i)\sum_{k=1}^i y(k) \bigg)^q a(m)\bigg)^{\frac{1}{q}} \notag\\
&\hskip+3cm\leq \mathcal{D}_1 \bigg(\sum_{m=1}^\infty y(m)^p m^{-p}B(m+1)^{p-1} B(m) b(m+1)^{1-p}\bigg)^{\frac{1}{p}}\label{cor1}\\
\intertext{	and}
&\bigg(\sum_{m=1}^{\infty}
\bigg(\sup_{m \le i < \infty} u(i) i \sum_{k=i}^\infty y(k)\bigg)^q a(m)\bigg)^{\frac{1}{q}}\notag\\ 
&\hskip+3cm\leq \mathcal{D}_2 \bigg(\sum_{m=1}^\infty y(m)^p B(m+1)^{p-1} B(m) b(m+1)^{1-p} \bigg)^{\frac{1}{p}} \label{cor2}
\end{align}
hold for all non-negative sequences $\bm{y}$. Moreover, if  $C$, $\mathcal{D}_1$ and $\mathcal{D}_2$ denote the best constants in \eqref{eq1}, \eqref{cor1}  and \eqref{cor2}, respectively, then 
\[
C\approx C_0 + \mathcal{D}_1 + \mathcal{D}_2,
\]
where $C_0$ is defined in \eqref{eq0} and the implicit multiplicative constants in the equivalence depend only on $p$ and $q$.
	
{\textup{(ii)}} $ 0<p \le 1$, and inequalities
\begin{align}
\bigg(\sum_{m=1}^{\infty}
\bigg(\sup_{m\le i<\infty} u(i)^p\sum_{k=1}^i y(k)\bigg)^{\frac{q}{p}} a(m)\bigg)^{\frac{1}{q}} &\leq \mathcal{D}_3 \bigg(\sum_{m=1}^\infty y(m) m^{-p} B(m)\bigg)^{\frac{1}{p}} \label{cor3}
\end{align}
and
\begin{align}
\bigg(\sum_{m=1}^{\infty}
\bigg(\sup_{m\le i<\infty} u(i)^p i^p \sum_{k=i}^\infty y(k)\bigg)^{\frac{q}{p}} a(m)\bigg)^{\frac{1}{q}} &\leq \mathcal{D}_4 \bigg(\sum_{m=1}^\infty y(m) B(m)\bigg)^{\frac{1}{p}} \label{cor4}
\end{align}
hold for all non-negative sequences $\bm{y}$.  
Moreover, if  $C$, $\mathcal{D}_3$ and $\mathcal{D}_4$ denote the best constants in \eqref{eq1}, \eqref{cor3}  and \eqref{cor4}, respectively, then 
\[
C\approx \mathcal{D}_3+\mathcal{D}_4,
\]
where the implicit multiplicative constants in the equivalence depend only on $p$ and $q$.
\end{cor}

\section{Auxiliary statements} \label{S4}

We present some auxiliary statements that will be used in the proof of the main theorems.

\begin{lem}(\cite[Lemma~1]{Bennett Gross 1})\label{Fubini}
Let $\bm{a}$ and $\bm{b}$ be given non-negative sequences. Then
\begin{equation*}
\sum_{m=1}^\infty a(m) \sum_{k=1}^m b(k)< \infty\quad \text{if and only if} \quad  \sum_{m=1}^\infty b(m) \sum_{k=m}^\infty a(k) <\infty.
\end{equation*}
Moreover, 
\begin{equation}\label{eqfubini}
\sum^\infty_{m=1} a(m) \sum_{k=1}^m b(k)= \sum_{m=1}^\infty b(m) \sum_{k=m}^\infty a(k).
\end{equation}
\end{lem}

\begin{lem}(\cite[Lemma~1]{Bennett Gross})
Let $0<p<\infty$ and $\bm{b}$ be a non-negative sequence. If $1\le p$, then
\begin{equation} \label{prl1}
\min(p,1)\sum^n_{k=1} b(k)B(k)^{p-1}\le B(n)^{p}\le \max(p,1)\sum^n_{k=1} b(k) B(k)^{p-1}
\end{equation}
holds for every $n \in \mathbb{N}$. If $ 0 < p < 1$ and $b(1)>0$, then \eqref{prl1} holds, as well. 
\end{lem}

\begin{lem}(\cite[Lemma~1']{Bennett Gross}) 
Let $0<p<\infty$ and $\bm{b}$ be a non-negative sequence.  If $1\le p$, then
\begin{equation} \label{prl2}
\min(p,1)\sum_{k=n}^{\infty}b(k) \bigg(\sum_{j=k}^{\infty} b(j)\bigg)^{p-1}\le \bigg(\sum_{k=n}^{\infty} b(k)\bigg)^{p}\le \max(p,1)\sum_{k=n}^{\infty} b(k) \bigg(\sum_{j=k}^{\infty} b(j)\bigg)^{p-1}
\end{equation}
holds for every $n \in \mathbb{N}$. If $0 < p < 1$ and $\bm{b}$ is a positive sequence, then  \eqref{prl2} holds, as well. 
\end{lem}

\begin{lem}(\cite[Lemma 5]{Bennett Gross})  \label{lemma_03}
Let $0<p\le 1$. Assume that $\bm{a}$ and $\bm{b}$ are given non-negative sequences. If 
\[ 
\bigg(\sum^m_{k=1}a(k)\bigg)^{p} \le \sum^m_{k=1}b(k), \quad m\in \mathbb{N}
\]
holds, then
\[ 
\bigg(\sum^m_{k=1} a(k) x(k) \bigg)^{p} \le\sum^m_{k=1}b(k) x(k)^{p}, \quad m\in \mathbb{N}\]
holds for all non-negative and non-increasing sequences $\bm{x}$.
\end{lem} 

A version of the following lemma was proved in \cite[Lemma~2.4]{BGKKT} for the case $B(\infty)=\infty$. To meet the requirements of the present paper, we employ a slightly modified version that accounts for the case where $B(\infty)$ may be finite, and we provide the proof for completeness.

\begin{lem}\cite[Lemma~2.4]{BGKKT} 
Let $0 < p < \infty$. Assume that $\bm{b}$ is a given non-negative sequence with $b(1)>0$. 
Then, for every $m \in \mathbb{N}$,
\begin{align} \label{prl2-}
\min\bigg(\frac{1}{p}, 1\bigg) \bigg( \frac{1}{B(m)^p} - \frac{1}{B(\infty)^p} \bigg)& \le \sum_{k=m}^{\infty} \frac{b(k+1)}{B(k)^p B(k+1)} \notag\\
& \le \max\bigg(\frac{1}{p}, 1\bigg) \bigg( \frac{1}{B(m)^p} - \frac{1}{B(\infty)^p} \bigg).
\end{align}
\end{lem}

\begin{proof}
For any $0 < a < b$ and $p > 0$, we have
\begin{equation*}
\min(p, 1) b^{p-1}(b-a) \le b^p - a^p \le \max(p, 1) b^{p-1}(b-a).
\end{equation*}  
By choosing $b= B(k)^{-1}$ and $a=B(k+1)^{-1}$, we obtain
\begin{equation*}
\min(p, 1) \frac{b(k+1)}{B(k)^p B(k+1)} \le \frac{1}{B(k)^p} - \frac{1}{B(k+1)^p} \le \max(p, 1) \frac{b(k+1)}{B(k)^p B(k+1)}.
\end{equation*}
By rearranging the terms, we have
\begin{align*}
\min\bigg(\frac{1}{p}, 1\bigg) \bigg( \frac{1}{B(k)^p} - \frac{1}{B(k+1)^p} \bigg) &\le \frac{b(k+1)}{B(k)^p B(k+1)} \\
&\le \max\bigg(\frac{1}{p}, 1\bigg) \bigg( \frac{1}{B(k)^p} - \frac{1}{B(k+1)^p} \bigg).
\end{align*}
Summing these estimates over $k$ from $m$ to $\infty$, we obtain the desired estimate \eqref{prl2-}.	
\end{proof}

The following lemma presents a variant of Copson's inequality (see \cite[Theorem~A and Theorem~B]{Cop}, \cite[Lemma~2, (28)]{Bennett Gross}, \cite[(11.1i) and (11.1ii)]{Grosse}). This particular formulation was stated in \cite{BGKKT} without proof, so we include its proof here for completeness.

\begin{lem} Let $1\leq p<\infty$. Assume that $\bm{b}$ is a given non-negative sequence with $b(1)>0$. Then inequality 
\begin{equation}  
\bigg(\sum_{m=1}^{\infty} 
\bigg( \sum^\infty_{k=m} x(k) \bigg)^p b(m) \bigg)^{\frac{1}{p}} \le p  2^{\frac{p-1}{p}} \bigg( \sum^\infty_{m=1} x(m)^p b({m+1})^{1-p} B(m) B({m+1})^{p-1} \bigg)^{\frac{1}{p}}
\label{eq12}
\end{equation}
holds for all non-negative sequences $\bm{x}$.
\end{lem}

\begin{proof}
For $p=1$, the result is trivial, as \eqref{eq12} reduces to an identity owing to Lemma~\ref{Fubini}. 
	
We now assume that $p>1$. Let $N\in \mathbb{N}$ be fixed. Applying \eqref{prl2} together with \eqref{eqfubini} yields
\begin{align*}
\sum_{m=1}^{N} \bigg(\sum_{k=m}^{N} x(k) \bigg)^p b(m) &\le p\sum_{m=1}^{N} \bigg(\sum^N_{k=m} x(k) \bigg(\sum^N_{i=k} x(i) \bigg)^{p-1}\bigg)  b(m)\\
& = p\sum_{k=1}^{N} x(k) \bigg(\sum^N_{i=k} x(i) \bigg)^{p-1} B(k) \\
& = p\sum_{k=1}^{N} x(k) B(k)^{\frac{1}{p}} B(k+1)^{\frac{p-1}{p}}b(k+1)^{\frac{1-p}{p}} \\
&\hskip+1cm\times B(k)^{1-\frac{1}{p}} B(k+1)^{\frac{1-p}{p}}b(k+1)^{\frac{p-1}{p}}\bigg(\sum^N_{i=k} x(i) \bigg)^{p-1}.
\end{align*}
Since $p>1$, H\"{o}lder's inequality gives
\begin{align}
\sum_{m=1}^{N} \bigg(\sum^N_{k=m} x(k) \bigg)^p b(m) 
&\le p\bigg(\sum_{k=1}^{N}
x(k)^p B(k) B(k+1)^{p-1}b(k+1)^{1-p}\bigg)^{\frac{1}{p}} \notag \\
&\hskip+1cm\times \bigg(\sum_{k=1}^{N}B(k) B(k+1)^{-1}b(k+1)\bigg(\sum^N_{i=k} x(i) \bigg)^{p}\bigg)^{1-\frac{1}{p}}. \label{mid-copson variant}
\end{align} 
Observe that, for each $n \in \mathbb{N}$, 
\begin{align*} 
\sum_{k=1}^{n} B(k) B(k+1)^{-1} b(k+1) &= \sum_{k=1}^{n-1} B(k) B(k+1)^{-1} b(k+1)+ B(n) B(n+1)^{-1}b(n+1)\\
&\le \sum_{k=1}^{n-1}b(k+1)+ B(n) \\
&\le 2B(n).
\end{align*} 
Thus, by Lemma~\ref{lemma_03}, we obtain 
\begin{align*} 
\sum_{k=1}^{N}B(k) B(k+1)^{-1}b(k+1) \bigg(\sum_{i=k}^N x(i) \bigg)^{p} &\le 2 \sum_{k=1}^N b(k)\bigg(\sum_{i=k}^N x(i) \bigg)^{p}.
\end{align*} 
By substituting this estimate into \eqref{mid-copson variant}, we arrive at
\begin{align} \sum_{m=1}^{N}
\bigg(\sum^N_{k=m} x(k) \bigg)^p b(m)&\le p 2^{1-\frac{1}{p}} \bigg(\sum_{k=1}^N x(k)^p B(k) B(k+1)^{p-1}b(k+1)^{1-p}\bigg)^{\frac{1}{p}}\notag\\
&\hskip+1cm \times \bigg( \sum_{k=1}^N b(k) \bigg( \sum^N_{i=k} x(i) \bigg)^{p} \bigg)^{1-\frac{1}{p}}. \label{approximate}
\end{align} 
After dividing both sides by the common factor, we get
\begin{align*} 
\bigg(\sum_{m=1}^{N}
\bigg(\sum^N_{k=m} x(k) \bigg)^p b(m)\bigg)^{\frac{1}{p}}&\le p 2^{1-\frac{1}{p}} \bigg(\sum_{k=1}^N x(k)^p B(k) B(k+1)^{p-1}b(k+1)^{1-p}\bigg)^{\frac{1}{p}}.
\end{align*} 
Thus, letting $N \to \infty$, \eqref{eq12} follows.
\end{proof}

The following results from \cite{GogKrep} are essential for proving our main theorems.

\begin{thm} \cite[Theorem~3.7]{GogKrep}\label{T:ekv-antigop}
Let	$0<r<\infty$. Assume that $\bm{v}$ and $\bm{w}$ are given non-negative sequences. Define
	\begin{align*}
		V_1 & := \sup_{\bm{a} \in  \RpN} \Bigg(\sum_{n=1}^{\infty}  w(n) \Big[ \sup_{n\le j <\infty} u(j)   \sup_{j\le i < \infty} a(i) \Big]^r \Bigg)^ {\frac{1}{r}} \Bigg( \sum_{n=1}^{\infty} v(n) a(n) \Bigg)^{-1},
        \\
		V_2 & := \sup_{\bm{a} \in  \RpN} \Bigg( \sum_{n=1}^{\infty} w(n) \Big[ \sup_{n \le j < \infty} u(j) \sum_{i=j}^{\infty} a(i) \Big]^r \Bigg)^{\frac{1}{r}} \Bigg( \sum_{n=1}^{\infty} v(n) a(n) \Bigg)^{-1}. 
	\end{align*}
	Then $V_1 \approx V_2$, where the multiplicative constants in the equivalence depend only on $r$.	
\end{thm}

\begin{thm}\cite[Theorem~3.8]{GogKrep}\label{T:ekv-gop}
	Let	$0< r < \infty$. $\bm{v}$ and $\bm{w}$ are given non-negative sequences. Define
	\begin{align*}
		S_1 & := \sup_{\bm{a} \in  \RpN}  \Bigg(\sum_{n=1}^{\infty}  w(n) \Big[ \sup_{n \le j < \infty} u(j)   \sup_{1 \le i \le j} a(i) \Big]^r \Bigg)^ {\frac{1}{r}}  \Bigg( \sum_{n=1}^{\infty} v(n) a(n) \Bigg)^{-1}, 
        \\
		S_2 &= \sup_{\bm{a} \in  \RpN}  \Bigg(\sum_{n=1}^{\infty}  w(n) \Big[ \sup_{n \le j < \infty} u(j) \sum_{i=1}^j a(i) \Big]^r \Bigg)^{\frac{1}{r}} \Bigg( \sum_{n=1}^{\infty} v(n) a(n) \Bigg)^{-1}. 
        \end{align*}
	Then $S_1 \approx S_2$, where the multiplicative constants in the equivalence depend only on $r$.	
\end{thm}	

\begin{thm}\cite[Theorem~2.1]{GogKrep} \label{T:main-discrete-p-big}
Let $0<p ,q <\infty$. Assume that $\bm{u}, \bm{v}$ and $\bm{w}$ are given non-negative sequences, and let $\bm{z}$ be defined by \eqref{def:z(m)}. Then, inequality 
\begin{equation}\label{E:d-gop}
\Bigg(\sum_{n=1}^{\infty} \Bigg(\sup_{n\le i < \infty} u(i) \sum_{k=1}^i x(k) \Bigg)^{q} w(n)\Bigg)^{\frac 1q} \le \mathcal{C}	\Bigg(\sum_{n=1}^{\infty} x(n)^p v(n) \Bigg)^{\frac 1p}
\end{equation}
holds for all non-negative sequences $\bm{x}$ if and only if one of the following conditions is satisfied:
	
{\textup{(i)}} $1 < p\le q$ and $\max\{\mathcal{C}_1, \mathcal{C}_2\} <\infty$, where
\begin{align*}
\mathcal{C}_1 &:= \sup_{1 \le n < \infty} z(n)\Bigg(\sum_{i=1}^{n} w(i)  \Bigg)^{\frac{1}{q}} \Bigg(\sum_{k=1}^{n} v(k)^{\frac{1}{1-p}} \Bigg)^{\frac{p-1}{p}},\\
\mathcal{C}_2 &:= \sup_{1 \le n < \infty} \Bigg( \sum_{i=n}^{\infty} z(i)^{q} w(i) \Bigg)^{\frac{1}{q}} \Bigg(\sum_{k=1}^{n} v(k)^{\frac{1}{1-p}} \Bigg)^{\frac{p-1}{p}}.
\end{align*}

{\textup{(ii)}} $1<p$, $q<p$ and $\max\{\mathcal{C}_3, \mathcal{C}_4\} < \infty$, where 
\begin{align*}
\mathcal{C}_3 &:= \Bigg( \sum_{n=1}^{\infty} \Bigg(\sum_{i=n}^{\infty}  z(i)^q w(i) \Bigg)^{\frac{q}{p-q}} z(n)^{q} w(n) \Bigg(\sum_{k=1}^{n} v(k)^{\frac{1}{1-p}} \Bigg)^{\frac{(p-1)q}{p-q}} \Bigg)^\frac{p-q}{pq},\\
\mathcal{C}_4	& :=  \Bigg(\sum_{n=1}^{\infty} \Bigg(\sum_{i=1}^{n} w(i) \Bigg)^{\frac{q}{p-q}} w(n) \sup_{n\le k <\infty} z(k)^{\frac{pq}{p-q}} \Bigg(\sum_{j=1}^{k} v(j)^{\frac{1}{1-p}} \Bigg)^{\frac{(p-1)q}{p-q}} \Bigg)^\frac{p-q}{pq}.
\end{align*}
	
{\textup{(iii)}} $0<p\le 1$, $p\le q$ and $\max\{\mathcal{C}_5, \mathcal{C}_6\} < \infty$, where
\begin{align*}
\mathcal{C}_5 & := \sup_{1\le n < \infty}  z(n)\Bigg( \sum_{i=1}^{n} w(i) \Bigg)^{\frac 1q} \sup_{1\le j\le n} v(j)^{-\frac{1}{p}},\\
\mathcal{C}_6 &:= \sup_{1\le n < \infty} \Bigg( \sum_{k=n}^{\infty} z(k)^q w(k) \Bigg)^{\frac 1q} \sup_{1\le j\le n} v(j)^{-\frac{1}{p}}.
\end{align*}
	
{\textup{(iv)}} $0<q<p\le 1$ and $\max\{\mathcal{C}_7, \mathcal{C}_8\}< \infty$, where
\begin{align*}
\mathcal{C}_7 &:=\Bigg( \sum_{n=1}^{\infty} \Bigg(\sum_{i=n}^{\infty} z(i)^q w(i) \Bigg)^{\frac{q}{p-q}} z(n)^{q} w(n) \sup_{1\le k\le n} v(k)^{-\frac{q}{p-q}} \Bigg)^{\frac{p-q}{pq}},\\
\mathcal{C}_8 & :=\Bigg(\sum_{n=1}^{\infty} \Bigg(\sum_{i=1}^{n} w(i) \Bigg)^{\frac{q}{p-q}} w(n) \sup_{n \le k< \infty} z(k)^{\frac{pq}{p-q}} v(k)^{-\frac{q}{p-q}} \Bigg)^{\frac{p-q}{pq}}.
\end{align*}
Moreover, if $\mathcal{C}$ denotes the best constant in inequality \eqref{E:d-gop}, then
\begin{equation*}
\mathcal{C} \approx
\begin{cases}
\mathcal{C}_1+\mathcal{C}_2 &\text{in the case \textup{(i)},} \\
\mathcal{C}_3+ \mathcal{C}_4 &\text{in the case \textup{(ii)},} \\
\mathcal{C}_5+ \mathcal{C}_6  &\text{in the case \textup{(iii)},} \\
\mathcal{C}_7+ \mathcal{C}_8  &\text{in the case \textup{(iv)},}
\end{cases}
\end{equation*}
where the multiplicative constants in all the equivalences above depend only on $p$ and $q$.
\end{thm}

The following theorem was originally presented in \cite[Theorem 2.2]{GogKrep}. However, since cases (ii) and (iii) in the original version contain certain misprints, we provide the corrected statement here by incorporating the results from \cite{GPU}.

\begin{thm}
\label{T:main-antigop}
Let $0<p ,q <\infty$. Assume that $\bm{u}, \bm{v}$ and $\bm{w}$ are given non-negative sequences. Then, inequality 
\begin{equation*}
\Bigg(\sum_{m=1}^{\infty} \Bigg(\sup_{m \le i < \infty} u(i) \sum_{k=i}^{\infty} x(k) \Bigg)^{q} w(m) \Bigg)^{\frac{1}{q}} \le \mathfrak{C} \Bigg(\sum_{m=1}^{\infty} x(m)^p v(m)\Bigg)^{\frac{1}{p}}
\end{equation*}
holds for all non-negative sequences $\bm{x}$ if and only if 
	
{\textup{(i)}} $1 < p\le q$ and $\mathfrak{C}_1 < \infty$, where
\begin{align*}
\mathfrak{C}_1  :=  \sup_{1 \le n < \infty} \Bigg( \sum_{i=1}^n w(i) \sup_{i\le j\le n} u(j)^q \Bigg)^{\frac{1}{q}} \Bigg(\sum_{k=n}^{\infty} v(k)^{\frac{1}{1-p}} \Bigg)^{\frac{p-1}{p}}.  
\end{align*}
	
{\textup{(ii)}}  $1< p$, $q<p$ and $\max\{\mathfrak{C}_2, \mathfrak{C}_3\} < \infty$, where
\begin{align*}
\mathfrak{C}_2  & := \Bigg( \sum_{n=1}^{\infty}  \Bigg(\sum_{i=1}^{n} w(i) \Bigg)^{\frac{q}{p-q}} w(n) \sup_{n \le k < \infty} u(k)^{\frac{pq}{p-q}} \Bigg(\sum_{m=k}^{\infty} v(m)^{\frac{1}{1-p}} \Bigg)^{\frac{(p-1)q}{p-q}} \Bigg)^\frac{p-q}{pq},\\
\mathfrak{C}_3 & := \Bigg( \sum_{n=1}^{\infty} \Bigg(\sum_{i=1}^{n} w(i) \sup_{i\le j\le n} u(j)^q \Bigg)^{\frac{q}{p-q}} w(n) \sup_{n \le k < \infty} u(k)^q \Bigg(\sum_{m=k}^{\infty} v(m)^{\frac{1}{1-p}} \Bigg)^{\frac{(p-1)q}{p-q}} \Bigg)^\frac{p-q}{pq}.
\end{align*}
	
{\textup{(iii)}} $0 < p \le 1$,  $p \le q$ and $\mathfrak{C}_4 <\infty$, where
\begin{align*}
\mathfrak{C}_{4} := \sup_{1\le n <\infty} \Bigg( \sum_{i=1}^n w(i)\sup_{i\le j\le n}u(j)^q  \Bigg)^{\frac 1q} \sup_{n\le j<\infty} v(j)^{- \frac{1}{p}}.
\end{align*}
	
{\textup{(iv)}}  $0 < q < p \le 1$ and $\max\{\mathfrak{C}_{5}, \mathfrak{C}_{6}\} <\infty$, where
\begin{align*}
\mathfrak{C}_{5}&:=   \Bigg( \sum_{n=1}^\infty \Bigg( \sum_{i=1}^n w(i) \Bigg)^{\frac{q}{p-q}} w(n) \sup_{n\le k< \infty} u(k)^{\frac{pq}{p-q}} \sup_{k\le m <\infty} v(m)^{-\frac q{p-q}}  \Bigg)^\frac{p-q}{pq},\\
\mathfrak{C}_{6}&:= \Bigg( \sum_{n=1}^\infty \Bigg(\sum_{i=1}^n w(i) \sup_{i\le j\le n} u(j)^q \Bigg)^{\frac{q}{p-q}} w(n) \sup_{n\le k <\infty } u(k)^q  \sup_{k\le m <\infty} v(m)^{-\frac{q}{p-q}} \Bigg)^{\frac{p-q}{pq}}.
\end{align*}
Moreover, if $\mathfrak{C}$ denotes the best constant in inequality \eqref{E:d-gop}, then
\begin{equation*}
\mathfrak{C}  \approx
\begin{cases}
\mathfrak{C}_1  &\text{in the case \textup{(i)},} \\
\mathfrak{C}_2+\mathfrak{C}_3&\text{in the case \textup{(ii)},} \\
\mathfrak{C}_4 &\text{in the case \textup{(iii)},} \\
\mathfrak{C}_5+\mathfrak{C}_6 &\text{in the case \textup{(iv)}},
\end{cases}
\end{equation*}
where the multiplicative constants in all the equivalences above depend only on $p$ and $q$.
\end{thm}

\section{Proofs}  \label{S5}

In this section, we present the proofs of the main results stated in Sections \ref{S3} and \ref{S2}.

\begin{proof}[{\bfseries Proof of Theorem~\ref{Th0}}]
The implication $\eqref{eq1}\implies\eqref{eq0}$ is straightforward: by taking $\bm{x}=(1,1,{\color{red}\ldots})$, we immediately obtain $C_0 \leq C$. 
	
Conversely, the validity of \eqref{eq0} yields
\begin{align*}
\bigg(\sum_{m=1}^{\infty} \bigg(\sup_{m\le i<\infty} u(i) \sum^i_{k=1}x(k)\bigg)^q a(m) \bigg)^{\frac{1}{q}} & \le x(1) \bigg(\sum_{m=1}^{\infty} \bigg(\sup_{m\le i<\infty} u(i) i \bigg)^q a(m) \bigg)^{\frac{1}{q}}\notag\\
&\hskip-5cm  = C_0
x(1)\bigg(\sum_{m=1}^{\infty} b(m) \bigg)^{\frac{1}{p}} \le C_0 b(1)^{-\frac{1}{p}} \bigg( \sum_{m=1}^{\infty} b(m) \bigg)^{\frac{1}{p}} \bigg(\sum_{m=1}^{\infty} x(m)^p b(m) \bigg)^{\frac{1}{p}}.
\end{align*}
Hence, inequality \eqref{eq1} holds with $C \leq C_0 b(1)^{-\frac{1}{p}} B(\infty)^{\frac{1}{p}}$.
\end{proof}

\begin{proof}[\bfseries Proof of Theorem \ref{Th1}] 
The necessity of the condition $C_0 < + \infty$ is established by testing inequality \eqref{eq1} with $\bm{x}=(1,1,\ldots)$. If $B(\infty) < \infty$, the definition of $C_0$ in \eqref{eq0} implies $C_0 \le C$; if $B(\infty) = \infty$, the inequality holds trivially as $C_0 = 0$. To prove the implication  $\eqref{eq1}\implies\eqref{eq2}$, let $\bm{y}$ be an arbitrary non-negative sequence. If the right-hand side of \eqref{eq2} is infinite, the inequality holds trivially. Therefore, we may assume that the right-hand side is finite. This implies that the right-hand side of \eqref{eq12} is also finite, which ensures $x(k)=\sum_{j=k}^\infty y(j)$ is finite for each $k \in \mathbb{N}$ and satisfies $\lim_{k\to \infty} x(k) = 0$. Taking such $\bm{x}$ in \eqref{eq1} and applying \eqref{eq12} yields \eqref{eq2}.
	
Now, for the converse, assume that $1 < p < \infty$, $C_0 < +\infty$ and inequality \eqref{eq2} holds for all non-negative sequences $\bm{y}$. Let $\bm{x}$ be a non-negative and non-increasing sequence. We first observe that, by using \eqref{prl1}, \eqref{prl2-} and the monotonicity of $\bm{x}$, we obtain
\begin{align*}
x(k)=&B(k)\bigg(\frac{1}{B(k)}-\frac{1}{B(\infty)}\bigg)x(k)+ \frac{B(k)}{B(\infty)}x(k) \\
&\le p^{\frac{1}{p}}\bigg(\sum_{m=1}^k B(m)^{p-1}b(m)x(m)^p\bigg)^{\frac{1}{p}}\sum_{i=k}^\infty \frac{b({i+1})}{B(i)B({i+1})}\notag\\
&\hskip+1cm + \frac{B(k)^{1-\frac{1}{p}}}{B(\infty)}\bigg(\sum_{m=1}^kx(m)^pb(m)\bigg)^{\frac{1}{p}}
\\
&\le p^{\frac{1}{p}} \sum_{i=k}^\infty \frac{b({i+1})}{B(i)B({i+1})} \bigg(\sum_{m=1}^i B(m)^{p-1}b(m)x(m)^p\bigg)^{\frac{1}{p}}\\
&\hskip+1cm + B(\infty)^{-\frac{1}{p}} \bigg(\sum_{m=1}^\infty x(m)^p b(m) \bigg)^{\frac{1}{p}}.
\end{align*}
	
By combining this estimate with the sub-additivity of the supremum, the well-known inequality 
\[(a+b)^{1/q}\le \max\bigg(2^{\frac{1}{q}-1}, 1\bigg) (a^{1/q}+b^{1/q}),\] 
and the validity of inequality \eqref{eq2} and \eqref{eq0}, yields
\begin{align*}
&\bigg(\sum_{m=1}^{\infty} \bigg(\sup_{m\le j<\infty} u(j)\sum_{k=1}^j x(k) \bigg)^q a(m) \bigg)^{\frac{1}{q}} \notag\\
& \hskip1cm \le p^{\frac{1}{p}} \max\big(2^{\frac{1}{q}-1}, 1\big) \bigg(\sum_{m=1}^{\infty} \bigg[\sup_{m\le j<\infty} u(j) \sum_{k=1}^j \sum_{i=k}^\infty \frac{b({i+1})}{B(i)B({i+1})} \\
& \hskip5.5cm \times \bigg(\sum_{m=1}^i B(m)^{p-1}b(m) x(m)^p \bigg)^{\frac{1}{p}} \bigg]^q a(m)\bigg)^{\frac{1}{q}} \notag\\
&\hskip1.5cm +  \max\big(2^{\frac{1}{q}-1}, 1\big) \bigg(\sum_{m=1}^{\infty} \bigg(\sup_{m \le j < \infty} u(j) j \bigg)^q a(m) \bigg)^{\frac{1}{q}} B(\infty)^{-\frac{1}{p}} \bigg(\sum_{m=1}^\infty x(m)^p b(m) \bigg)^{\frac{1}{p}}\\
& \hskip1cm \le \mathcal{D} p^{\frac{1}{p}} \max\big(2^{\frac{1}{q}-1}, 1\big) 
\bigg(\sum_{i=1}^{\infty} \bigg(\sum_{m=1}^i B(m)^{p-1} b(m) x(m)^p \bigg)\frac{b({i+1})}{B(i)^{p-1} B({i+1})} \bigg)^{\frac{1}{p}}\\
&\hskip1.5cm + C_0 \max\big(2^{\frac{1}{q}-1}, 1\big) \bigg(\sum_{m=1}^\infty x(m)^p b(m) \bigg)^{\frac{1}{p}}.
\end{align*}
Applying \eqref{eqfubini} and \eqref{prl2-}, we obtain
\begin{align*}
&\bigg(\sum_{m=1}^{\infty} \bigg(\sup_{m\le j<\infty} u(j)\sum_{k=1}^j x(k)\bigg)^qa(m)\bigg)^{\frac{1}{q}}\notag\\
& \hskip1cm \le \mathcal{D} p^{\frac{1}{p}} \max\big(2^{\frac{1}{q}-1}, 1\big) \bigg(\sum_{m=1}^{\infty} B(m)^{p-1} b(m) x(m)^p \sum_{i=m}^\infty \frac{b({i+1})} {B(i)^{p-1}B({i+1})}\bigg)^{\frac{1}{p}}\notag\\
&\hskip+1cm + C_0 \max\big(2^{\frac{1}{q}-1}, 1\big)\bigg(\sum_{m=1}^\infty x(m)^p b(m) \bigg)^{\frac{1}{p}}\\
&\hskip1cm \le \bigg(C_0+\mathcal{D} p^{\frac{1}{p}}\max\bigg(\frac{1}{p-1},1\bigg)^{\frac{1}{p}}\bigg) \max\big(2^{\frac{1}{q}-1}, 1\big) \bigg(\sum_{i=1}^{\infty} x(m)^p b(m) \bigg)^{\frac{1}{p}}.
\end{align*}
Consequently, \eqref{eq1} holds with $C \le  \max\bigg(1,  p^{\frac{1}{p}}  \max\big(\frac{1}{p-1},1\big)^{\frac{1}{p}}\bigg) \max\big(2^{\frac{1}{q}-1}, 1\big)\bigg(\mathcal{D}+C_0\bigg)$.
\end{proof}

\begin{proof}[{\bfseries Proof of Theorem~\ref{Th2}}] 
Let $\bm{y}$ denote a non-negative sequence. It is easy to see that
\begin{align*} 
\sup_{m\le i<\infty} u(i) \bigg(\sup_{1\le k\le i} k^p\sup_{k\le j\le \infty} y(j)\bigg)^{\frac{1}{p}} &  \le \sup_{m\le i<\infty} u(i)\sup_{1\le k\le i}k \bigg(\sum_{j=k}^\infty y(j)\bigg)^{\frac{1}{p}}\notag \\
&\le \sup_{m\le i<\infty}u(i) \sum_{k=1}^i\bigg(\sum_{j=k}^\infty y(j)\bigg)^{\frac{1}{p}}.
\end{align*}
Therefore, we have 
\[
\textup{(ii)}\implies \textup{(iv)}\implies \textup{(vi)}.
\]
	
It is also apparent that 
\begin{align} \label{est2}
\sup_{m\le i<\infty} u(i) \bigg(\sup_{1\le k\le i} k^p\sup_{k\le j < \infty} y(j)\bigg)^{\frac{1}{p}}  & \le \sup_{m\le i<\infty}u(i) \sum_{k=1}^i\bigg(\sup_{k\le j< \infty} y(j)\bigg)^{\frac{1}{p}}\notag\\
&\le \sup_{m\le i<\infty}u(i) \sum_{k=1}^i\bigg(\sum_{j=k}^\infty y(j)\bigg)^{\frac{1}{p}}. 
\end{align}
As a result of \eqref{est2}, we get
\[
\textup{(ii)}\implies \textup{(v)}\implies \textup{(vi)}.
\]

Next, we establish the implication $\textup{(iii)}\implies \textup{(ii)}$. Note that for $p=1$, statements $\textup{(ii)}$ and  $\textup{(iii)}$ coincide. For $p<1$, using \eqref{prl1}, observe that, for any $m\in \mathbb{N}$,
\begin{equation}\label{GPSL-condition}
\bigg(\sum_{k=1}^m 1\bigg)^p \leq  \sum_{k=1}^m k^{p-1}    
\end{equation}
holds. Then, applying Lemma~\ref{lemma_03} with $a(k)=1$, $b(k)=k^{p-1}$ and $x(k)= \big(\sum_{j=k}^\infty y(j)\big)^{\frac{1}{p}}$, we obtain
\[\sum_{k=1}^i\bigg(\sum_{j=k}^\infty y(j)\bigg)^{\frac{1}{p}} \le  \bigg(\sum_{k=1}^i k^{p-1}\sum_{j=k}^\infty y(j)\bigg)^{\frac{1}{p}}, \quad i\in \mathbb{N}.
\]
Multiplying both sides by $u(i)$ and taking the supremum, we get
\begin{equation}\label{GPSL-result}
\sup_{m\le i<\infty}u(i)\sum_{k=1}^i\bigg(\sum_{j=k}^\infty y(j)\bigg)^{\frac{1}{p}} \le  \sup_{m\le i<\infty}u(i)\bigg(\sum_{k=1}^i k^{p-1}\sum_{j=k}^\infty y(j)\bigg)^{\frac{1}{p}}.   
\end{equation}
Thus, \eqref{GPSL-result} yields 
\[
\textup{(iii)}\implies \textup{(ii)}.
\]
Now, we show that $\textup{(vi)} \implies \textup{(iii)}$. We have
\begin{align*}
&\sup_{m\le i<\infty} u(i)^p \sup_{1\le k\le i} k^p \sup_{k\le j < \infty} y(j) \\
& = \sup_{m\le i<\infty} u(i)^p \max\bigg\{\sup_{1\le k\le i} k^{p} \sup_{k\le j\le i} y(j), \; i^p \sup_{i\le j < \infty} y(j)\bigg\} \\
& = \sup_{m\le i<\infty} u(i)^p  \max\bigg\{\sup_{1\le {\color{red}j}\le i} j^{p} y(j), \; i^{p} \sup_{i\le j < \infty} y(j)\bigg\}\\
& = \max\bigg\{\sup_{m\le i<\infty} u(i)^p \sup_{1\le {\color{red}j}\le i} j^p y(j), \sup_{m\le i<\infty} u(i)^p i^p \sup_{i\le j < \infty} y(j)\bigg\}.
\end{align*}
Thus, if inequality \eqref{eq7} holds, then
\begin{equation} \label{for6-1}
\bigg(\sum_{m=1}^\infty\bigg(\sup_{m\le i<\infty} u(i)^p \sup_{1\le j\le i} j^p  y(j)\bigg)^{\frac{q}{p}} a(m)\bigg)^{\frac{1}{q}} \leq \D_5\bigg(\sum_{m=1}^\infty y(m) B(m) \bigg)^{\frac{1}{p}} 
\end{equation}
and
\begin{equation} \label{for6-2}
\bigg(\sum_{m=1}^\infty\bigg( \sup_{m\le i<\infty} u(i)^p  i^p \sup_{i\le j < \infty} y(j)\bigg)^{\frac{q}{p}} a(m)\bigg)^{\frac{1}{q}} \leq \D_5\bigg(\sum_{m=1}^\infty y(m) B(m) \bigg)^{\frac{1}{p}} 
\end{equation}
hold, as well. Using Theorem~\ref{T:ekv-gop} with $r=q/p$, more precisely, the equivalence $S_1\approx S_2$, the validity of inequality \eqref{for6-1} yields
\begin{equation} \label{for6-1-1}
\bigg(\sum_{m=1}^\infty\bigg(\sup_{m\le i<\infty} u(i)^p \sum_{j=1}^i j^p  y(j)\bigg)^{\frac{q}{p}} a(m)\bigg)^{\frac{1}{q}} \lesssim \bigg(\sum_{m=1}^\infty y(m) B(m) \bigg)^{\frac{1}{p}},
\end{equation}
and analogously, using Theorem~\ref{T:ekv-antigop} with $r=q/p$, more precisely, the equivalence $V_1 \approx V_2$, the validity of inequality \eqref{for6-2} yields
\begin{equation} \label{for6-2-1}
\bigg(\sum_{m=1}^\infty\bigg( \sup_{m\le i<\infty} u(i)^p  i^p \sum_{j=i}^\infty y(j)\bigg)^{\frac{q}{p}} a(m)\bigg)^{\frac{1}{q}} \lesssim \bigg(\sum_{m=1}^\infty y(m) B(m) \bigg)^{\frac{1}{p}}. 
\end{equation}
	
Furthermore, by using Lemma~\ref{Fubini} and the fact that $\sum_{k=1}^i k^{p-1}\le \frac{1}{p}i^p$, which easily follows from the first estimate of Lemma~\ref{prl1} by taking $b(k)=1$, $k=1,\ldots,i$, we get 
\begin{align} \label{Fubini-result}
\sum_{k=1}^i k^{p-1}\sum_{j=k}^\infty y(j) 
& = \sum_{k=1}^i k^{p-1} \sum_{j=k}^i y(j)+ \sum_{k=1}^i k^{p-1} \sum_{j=i+1}^\infty y(j) \notag\\
& = \sum_{j=1}^i y(j) \sum_{k=1}^j k^{p-1} + \sum_{k=1}^i k^{p-1}\sum_{j=i+1}^\infty y(j) \notag\\
& \le \frac{1}{p} \bigg( \sum_{j=1}^i j^p y(j) +  i^{p} \sum_{j=i}^\infty y(j) \bigg).
\end{align}
Therefore, in view of \eqref{Fubini-result}, we have
\begin{align*}
\sup_{m\le i<\infty}u(i) \bigg(\sum_{k=1}^i k^{p-1}\sum_{j=k}^\infty y(j)\bigg)^{\frac{1}{p}} 
&\le  p^{-\frac{1}{p}} \sup_{m\le i<\infty} u(i)\bigg(\sum_{j=1}^i j^p y(j) +  i^p \sum_{j=i}^\infty y(j)\bigg)^{\frac{1}{p}}.
\end{align*}
Using the sub-additivity of the supremum, we have
\begin{align}\label{est1}
&\sup_{m\le i<\infty}u(i) \bigg(\sum_{k=1}^i k^{p-1}\sum_{j=k}^\infty y(j)\bigg)^{\frac{1}{p}}\notag \\
& \hskip+1cm \le 2^{\frac{1}{p}-1}p^{-\frac{1}{p}}\bigg(\sup_{m\le i<\infty} u(i) \bigg(\sum_{j=1}^i j^p y(j)\bigg)^{\frac{1}{p}} + \sup_{m\le i<\infty} u(i) i \bigg(\sum_{j=i}^\infty y(j)\bigg)^{\frac{1}{p}}\bigg).
\end{align}
Assuming that inequality \eqref{eq7} holds, it follows that \eqref{for6-1-1} and \eqref{for6-2-1} are also satisfied. Therefore, taking \eqref{est1} into consideration, we conclude that \eqref{eq4} holds, which yields
\[
\textup{(vi)}\implies \textup{(iii)}.
\]
To complete the proof, we establish the equivalence $\textup{(i)} \iff \textup{(ii)}$ by showing that \eqref{eq1} holds for all non-negative non-increasing sequences $\bm{x}$ if and only if \eqref{eq3} holds for all non-negative sequences $\bm{y}$.

Let $\bm{y}$ be an arbitrary non-negative sequence and assume that \eqref{eq1} holds for all non-negative, non-increasing sequences $\bm{x}$. Without loss of generality, assume that the right-hand side of \eqref{eq3} is finite, which means the sequence $x(k) = \big(\sum_{j=k}^{\infty} y(j)\big)^{1/p}$, $k \in \mathbb{N}$ is finite, non-negative, non-increasing, and satisfies $\lim_{k \to \infty} x(k) = 0$. Testing inequality \eqref{eq1} with $x(k)=\big(\sum_{j=k}^{\infty} y(j)\big)^{1/p}$ and applying Lemma~\ref{Fubini} for the right-hand side, we obtain $\textup{(ii)}$. Then, we have
\[
\textup{(i)} \implies \textup{(ii)}.
\]
Now, we will show the implication $\textup{(ii)} \implies \textup{(i)}$.

Testing inequality \eqref{eq3} with  the sequence $\bm{y}$ defined as  $y(k)=0$ if $k\not=n$ and $y(n)=1$, we obtain that 
\[ \bigg(\sum_{k=1}^n \bigg(\sup_{k\le i\le n}u(i)i\bigg)^q a(k) \bigg)^{\frac{1}{q}}\le \D_1 B(\infty)^{\frac{1}{p}}.
\]
Letting $n\to+\infty$, we get 
\begin{equation} \label{last}
\bigg(\sum_{k=1}^\infty \bigg(\sup_{k\le i\le \infty}u(i)i\bigg)^q a(k) \bigg)^{\frac{1}{q}}\le \D_1 B(\infty)^{\frac{1}{p}}. 
\end{equation}

On the other hand, let $\bm{x}$ be an arbitrary non-negative and non-increasing sequence and let $x(\infty):= \lim_{k\to \infty} x(k)$. Define the sequence  $y(k)=x(k)^p-x(k+1)^p$ for all $k \in \mathbb{N}$. If $x(\infty) = 0$, substituting  $y(k)$ in  \eqref{eq3}, and using \eqref{Fubini}, yields \eqref{eq1}. 

Now, assume that $x(\infty)> 0$. If $B(\infty)= \infty$, then the right-hand side of \eqref{eq1} is infinite, therefore inequality \eqref{eq1} holds trivially. Let us now assume that $B(\infty) <\infty$. Since  $y(k)=x(k)^p-x(k+1)^p$, using the equality $x(k)^p = \sum_{j=k}^\infty y(j) + x(\infty)^p$, we have
\begin{equation}\label{last_xk}
x(k) = \bigg(x(k)^p \bigg)^{\frac{1}{p}} \le 2^{\frac{1}{p}-1}\bigg(\sum_{j=k}^\infty y(j)\bigg)^{\frac{1}{p}}+2^{\frac{1}{p}-1}x(\infty).
\end{equation}

Therefore, combining \eqref{last_xk} with \eqref{eq3} and \eqref{last}, and invoking \eqref{Fubini} alongside the monotonicity of $\bm{x}$, we arrive at
\begin{align*} 
&\bigg(\sum_{m=1}^\infty \bigg(\sup_{m\le i < \infty} u(i)\sum_{k=1}^i x(k)\bigg)^q a(m) \bigg)^{\frac{1}{q}}\\
&\hskip+2cm\le 2^{\frac{1}{p}-1} \max\bigg(2^{\frac{1}{q}-1}, 1\bigg) \bigg(\sum_{m=1}^\infty \bigg(\sup_{m\le i< \infty} u(i)\sum_{k=1}^i \bigg(\sum_{j=k}^\infty y(j)\bigg)^{\frac{1}{p}} \bigg)^q a(m) \bigg)^{\frac{1}{q}}\\
&\hskip+3cm+2^{\frac{1}{p}-1}\max\bigg(2^{\frac{1}{q}-1}, 1\bigg)\bigg(\sum_{m=1}^\infty \bigg(\sup_{m\le i < \infty} u(i) \sum_{k=1}^i x(\infty)\bigg)^q a(m) \bigg)^{\frac{1}{q}}\\
&\hskip+2cm \le2^{\frac{1}{p}-1}\max\bigg(2^{\frac{1}{q}-1}, 1\bigg) \D_1\bigg(\sum_{k=1}^\infty (x(k)^p-x(\infty)^p) b(k)\bigg)^{\frac{1}{p}} \\
&\hskip+3cm+2^{\frac{1}{p}-1}\max\bigg(2^{\frac{1}{q}-1}, 1\bigg) \D_1 x(\infty)\bigg(\sum_{k=1}^\infty b(k)\bigg)^{\frac{1}{p}}\\
&\hskip+2cm\le 2^{\frac{1}{p}}\max\bigg(2^{\frac{1}{q}-1}, 1\bigg) \D_1\bigg(\sum_{k=1}^\infty x(k)^pb(k)\bigg)^{\frac{1}{p}}, 
\end{align*}
which yields the desired estimate \eqref{eq1}.
\end{proof}

\begin{proof}[{\bfseries Proof of Corollary~\ref{cor2.7}}]

{\textup(i)}   By Theorem~\ref{Th1},  inequality \eqref{eq1} holds for all non-negative non-increasing sequences $\bm{x}$ if and only if $C_0 < +\infty$, where $C_0$ is defined in \eqref{eq0}, and \eqref{eq2} holds for all non-negative sequences $\bm{y}$.

Since \eqref{eqfubini} gives 
\begin{align*}
\max &\bigg\{\sup_{n\le i<\infty}u(i)\sum_{j=1}^i j y(j) , \sup_{n\le i<\infty} u(i) i \sum_{j=i}^\infty y(j) \bigg\} \\
& \qquad \leq \sup_{n\le i<\infty}u(i)\sum_{k=1}^i\sum_{j=k}^\infty y(j)  \leq \sup_{n\le i<\infty}u(i)\sum_{j=1}^i j y(j) + \sup_{n\le i<\infty}u(i) i \sum_{j=i}^\infty y(j),
\end{align*}
we obtain 
\[ \eqref{eq2} \Leftrightarrow \big(\eqref{cor1} \,\& \, \eqref{cor2}\big).\]
	
{\textup(ii)} By Theorem~\ref{Th2}, \eqref{eq1} holds for all non-negative non-increasing sequences $\bm{x}$ if and only if \eqref{eq4} holds for all non-negative sequences $\bm{y}$.

Moreover, observe that using \eqref{eqfubini} and \eqref{GPSL-condition}, we have 
\begin{align}\label{revest1}
\sup_{m\le i<\infty} &u(i) \bigg(\sum_{k=1}^i k^{p-1}\sum_{j=k}^\infty y(j)\bigg)^{\frac{1}{p}}\notag \\
&\; \ge \max\bigg\{\sup_{m\le i<\infty}u(i) \bigg(\sum_{k=1}^i k^{p-1}\sum_{j=k}^i y(j)\bigg)^{\frac{1}{p}}, \sup_{m\le i<\infty}u(i) \bigg(\sum_{k=1}^i k^{p-1}\sum_{j=i}^\infty y(j)\bigg)^{\frac{1}{p}} \bigg\}\notag \\
&\; \ge \max\bigg\{\sup_{m\le i<\infty}u(i) \bigg(\sum_{j=1}^i y(j) \sum_{k=1}^j k^{p-1} \bigg)^{\frac{1}{p}}, \sup_{m\le i<\infty}u(i) \bigg(\sum_{k=1}^i k^{p-1}\sum_{j=i}^\infty y(j)\bigg)^{\frac{1}{p}} \bigg\}\notag \\
&\; \ge \max\bigg\{\sup_{m\le i<\infty}u(i) \bigg(\sum_{j=1}^i y(j) j^p \bigg)^{\frac{1}{p}}, \sup_{m\le i<\infty}u(i) i\bigg(\sum_{j=i}^\infty y(j)\bigg)^{\frac{1}{p}} \bigg\}
\end{align}
holds. Therefore, using \eqref{est1} and \eqref{revest1}, we arrive at 
\[ 
\eqref{eq4} \Leftrightarrow \big(\eqref{cor3} \, \& \, \eqref{cor4}\big).
\]
\end{proof}

\begin{proof}[{\bfseries Proof of Theorem~\ref{thmmain}}]

The proof follows easily from   Corollary~\ref{cor2.7} combined with Theorem~\ref{T:main-discrete-p-big}, Theorem~\ref{T:main-antigop} and \eqref{prl2-}. 

In case (i), we obtain $C \approx C_0 +C_1+C_2+C_3'$, where  
\begin{equation}
C_3' :=\bigg(\sup_{1\le n<\infty}  \sum_{i=1}^{n} a(i) \sup_{i\le j\le n} u(j)^q j^q \bigg)^{\frac{1}{q}} \bigg( B(n)^{-\frac{1}{p-1}} -B(\infty) ^{-\frac{1}{p-1}}\bigg)^{\frac{p-1}{p}}.
\end{equation}
Observe that
\begin{align*}
C_3'& \leq \sup_{1\le n<\infty}  \bigg(\sum_{i=1}^{n} a(i) \sup_{i\le j\le n} u(j)^q j^q \bigg)^{\frac{1}{q}} B(n)^{-\frac{1}{p}} = C_3 \lesssim C_0+ C_3'.
\end{align*}
Using this estimate, we can replace $C_3'$ with $C_3$ to deduce that $C \approx C_0+C_1+C_2+C_3$. Case (ii) can be proved similarly. 

In case (iv), the application of Corollary~\ref{cor2.7} together with Theorem~\ref{T:main-discrete-p-big}, Theorem~\ref{T:main-antigop}, we obtain $C\approx C_{11}+ C_{12}+{C_{13}}+ C_{13}'$, where
\begin{equation*}
C_{13}' : =\Bigg( \sum_{n=1}^{\infty}  A(n)^{\frac{q}{p-q}} a(n) \sup_{n \le k < \infty} u(k)^{\frac{pq}{p-q}} k^{\frac{pq}{p-q}} B(k)^{-\frac{q}{p-q}} \Bigg)^\frac{p-q}{pq}.
\end{equation*}
It is clear that $C_{13}'\leq C_{12}$. Thus, we obtain $C\approx C_{11}+ C_{12}+{C_{13}}$.	
\end{proof}


\begin{thebibliography}{99}
	
\bibitem{Bennett Gross}  
G.~Bennett, and K.-G.~Grosse-Erdmann, On series of positive terms. \emph{Houston J. Math.} {\textbf{31}} (2005), no.~2, 541--586.
	
\bibitem{Bennett Gross 1}
G.~Bennett, and K.-G.~Grosse-Erdmann, Weighted Hardy inequalities for decreasing sequences and functions.  \emph{Math. Ann.} \textbf{334} (2006), no.~3, 489--531.

\bibitem{BGKKT}	N.~Bokayev, A.~Gogatishvili,  G.~Karshygina, N.~Kuzeubayeva, and T.~ \"Unver, Reduction theorems for the discrete Hardy operator on the cones of monotone sequences. \emph{J. Math. Sci.} \textbf{291} (2025), no. 2, 217--223.

\bibitem{CKOP} A.~Cianchi, R.~Kerman, B.~Opic, and  L.~Pick. A sharp rearrangement inequality for fractional maximal operator. \emph{Studia Math.} \textbf{138} (2000), no.~3,  277--284.

\bibitem{Cop} E.T.~Copson, Note on series of positive terms.  \emph{J. London Math. Soc.} \textbf{2} (1927), 9--12.
	
\bibitem{GogKrep} 
A.~Gogatishvili, M.~K\v repela, R.~Ol'hava, and L.~Pick, Weighted inequalities for discrete iterated Hardy operators. \emph{Mediterr. J. Math.} \textbf{17} (2022), no.~4, 132.
	
\bibitem{GogMus-MIA} 
A.~Gogatishvili, and R.~Mustafayev, Iterated Hardy-type inequalities involving suprema. \emph{Math. Inequal. Appl.} \textbf{20} (2017), no.~4, 901--927.

\bibitem{GOP} A.~Gogatishvili, B.~Opic, and L.~Pick, Weighted inequalities for Hardy-type operators involving suprema. \emph{Collect. Math.} {\textbf 57} (2006), no.~3, 227--255.

\bibitem{GogPick-nrt} 
A.~Gogatishvili, and L.~Pick, A reduction theorem for supremum operators. \emph{J. Comput. Appl. Math.} \textbf{208} (2007), no.~1, 270--279.

\bibitem{GPU}
A.~Gogatishvili, L.~Pick, and T. \"{U}nver, Weighted inequalities for discrete iterated kernel
operators. \emph{Math. Nachr.} \textbf{295} (2022), no.~11, 1--26.

\bibitem{GogStep1} 
A.~Gogatishvili, and V.D.~Stepanov, Reduction theorems for operators on the cones of monotone functions. \emph{J. Math. Anal. Appl.} \textbf{405} (2013), no.~1, 156--172.
	
\bibitem{GogStep2} 
A.~Gogatishvili, and V.D.~Stepanov, Reduction theorems for weighted integral inequalities on the cone of monotone functions. \emph{Russ. Math. Surv.} \textbf{64} (2013), no.~4, 597--664.
	
\bibitem{Grosse} 
K.-G.~Grosse-Erdmann,
\emph{The blocking technique, weighted mean operators and Hardy's inequality.} Lecture Notes in Mathematics, no.~1679, Berlin, 1998.
	
	
\bibitem{OinarShalg} 
R.~Oinarov, and S.Kh.~Shalginbaeva, Weighted Hardy inequalities on the cone of monotone sequences. \emph{Izv. Minister. Nauki Akad. Nauk Resp. Kaz. Ser. Fiz.-Mat.} (1998), no.~1, 33--42.
	
\bibitem{Sawyer} 
E.~Sawyer, Boundedness of classical operators on classical Lorentz spaces. \emph{Studia Math.} \textbf{96} (1990), no.~2, 145--158.
	
\bibitem{Sinnamon} 
G.~Sinnamon, Hardy's inequality and monotonicity, Function Spaces and Nonlinear Analysis. Mathematical Institute of the Academy of Sciences of the Czech Republic, Prague (2005), 292--310 \url{https://www.math.uwo.ca/faculty/sinnamon/pdf/fsdona.pdf}
	

\end{thebibliography}
\end{document}